\documentclass[10pt, reqno]{amsart}
\usepackage[dvipsnames]{xcolor}
\usepackage{amsmath}
\usepackage{amsfonts}
\usepackage{amssymb, cite}
\usepackage{amsthm}
\usepackage{tikz}
\usepackage[sc]{mathpazo}
\usepackage{bm}
\usepackage{graphicx}

\usepackage{mathtools}
   \mathtoolsset{showonlyrefs}

\usepackage{algorithm}
\usepackage{algpseudocode}

\usepackage{enumitem}
\setenumerate{itemsep=5pt} 
\setitemize{itemsep=5pt} 

\usepackage[font=small,labelfont=bf]{caption}

\theoremstyle{plain}
\newtheorem{theorem}{Theorem}
\numberwithin{theorem}{section}
\newtheorem{proposition}[theorem]{Proposition}
\newtheorem{lemma}[theorem]{Lemma}
\newtheorem{corollary}[theorem]{Corollary}
\newtheorem*{conjecture}{Conjecture}
\theoremstyle{definition}

\newtheorem{problem}[theorem]{Problem}

\newtheorem{remark}[theorem]{Remark}

\numberwithin{equation}{section}

\newcommand{\Z}{\mathbb{Z}}

\newcommand{\M}{\mathrm{M}}

\newcommand{\N}{\mathbb{N}}
\newcommand{\Q}{\mathbb{Q}}

\newcommand{\s}{\mathfrak{s}}

\newcommand{\T}{\mathsf{T}}

\newcommand{\PPM}{\scriptstyle+ + -}

\newcommand{\mdim}{\operatorname{dim}_{\mathrm{mat}}}
\renewcommand{\pmod}[1]{\,\,(\operatorname{mod} #1)}

\newcommand{\LL}{\mathcal{L}}
\renewcommand{\S}{\mathcal{S}}

\newcommand{\tr}{\operatorname{tr}}
\newcommand{\lcm}{\operatorname{lcm}}

\newcommand{\ceiling}[1]{\lceil #1 \rceil}

\newcommand{\tinymatrix}[4]{\big[ \begin{smallmatrix} #1 & #2 \\ #3 & #4 \end{smallmatrix} \big]}
\newcommand{\minimatrix}[4]{ \begin{bmatrix} #1 & #2 \\ #3 & #4 \end{bmatrix} }

\newcommand{\semigroup}[1]{\langle #1 \rangle}

\allowdisplaybreaks
\usepackage{hyperref}

\begin{document}

\title[Numerical semigroups from rational matrices III]{Numerical semigroups from rational matrices III:  semigroups of matricial dimension two and a counterexample to the lonely element conjecture}
\author[A.~Chhabra]{Arsh Chhabra}
\address{Operations Research \& Financial Engineering, Princeton University, Sherrerd Hall, Charlton Street, Princeton, NJ 08544, USA}
\email{ac3019@princeton.edu}

\author[S.R.~Garcia]{Stephan Ramon Garcia}
\address{Department of Mathematics and Statistics, Pomona College, 610 N. College Ave., Claremont, CA 91711, USA}
\email{stephan.garcia@pomona.edu}
\urladdr{\url{https://stephangarcia.sites.pomona.edu/}}

\thanks{SRG was partially supported by NSF grant DMS-2054002.}


\maketitle

\begin{abstract}
We characterize semigroups in $\{0,1,2,\ldots\}$ of matricial dimension $2$ and produce a counterexample to the conjecture that a numerical semigroup whose small elements are lonely has matricial dimension at most $2$.
\end{abstract}

\section{Introduction}

Our aim in this paper is twofold.  First we characterize semigroups of matricial dimension $2$ \cite{NSRM1,NSRM2}, an endevaor that requires an analysis of the $p$-adic valuations of Lucas sequences \cite{BallotPaper, BallotBook}.  We use this knowledge
to provide a counterexample to the lonely element conjecture. In particular, we prove that there is a semigroup of matricial dimension
greater than $2$ that does not contain any non-lonely small elements.  We review the necessary terminology and notation below.

In this paper, the term \emph{semigroup} refers to an additive subsemigroup of $\N = \{0,1,2,\ldots\}$
(to the delight of logicians and the consternation of number theorists, we include $0$ as a natural number).
A \emph{numerical semigroup} is a semigroup in $\N$ with finite complement \cite{Assi, Rosales}.  The notation
$\semigroup{\cdot}$ denotes the semigroup generated by the indicated set.  For example, 
$\semigroup{2,7} = \{0,2,4,6,7,8,9,10\ldots\}$. 
The largest natural number not in $S$ is the \emph{Frobenius number} $g(S)$ of $S$.
Each numerical semigroup $S$ has a unique minimal system of \emph{generators} 
$1 \leq n_1 < n_2 < \cdots < n_k$ 
such that $\gcd(n_1,n_2,\ldots,n_k) = 1$ and
$S = \semigroup{n_1,n_2,\ldots,n_k}$ \cite[Thm.~2.7]{Rosales}.
Here $e(S) = k$ is the \emph{embedding dimension} of $S$ and $m(S) = n_1$ is the \emph{multiplicity} of $S$.
A nonzero element of a numerical semigroup $S$ is \emph{small} if it is less than the Frobenius number $g(S)$.
The set of nonzero small elements is
\begin{equation}\label{eq:SmallS}
    \s(S)=\{n\in S \setminus \{0\} :  n<  g(S)\}.
\end{equation}
We say $n \in \s(S)$ is \emph{lonely} if $n-1$ and $n+1$ do not belong to $S$.

A semigroup of the form
$\semigroup{m} = m\N = \{0,m,2m,3m,\ldots\}$ is a \emph{cyclic} semigroup.  In particular,
the \emph{trivial} semigroups $\{0\} = \semigroup{0}$ and $\N = \semigroup{1}$ are cyclic.
An \emph{ordinary} semigroup is a numerical semigroup of the form
\begin{equation}\label{eq:TailDefinition}
S_m = \{0,m,m+1,m+2,\ldots\}.
\end{equation}
We may refer to such a semigroup as a \emph{tail} (semigroup).
A \emph{scaled tail} is a semigroup of the form
$S_m(a;n) =\{ x \in S_m : x \equiv a \pmod{n} \}$. If $a=0$, we write 
\begin{equation}\label{eq:ModularTail}
S_m(n) =\semigroup{n}\cap S_m.
\end{equation}

In what follows, $\Q$ denotes the set of rational numbers and $\Z$ the set of integers.
The \emph{exponent semigroup} associated to $A \in \M_d(\Q)$ is
\begin{equation*}
\S(A) =  \{n \in \N : A^n \in \M_d(\Z)\}.
\end{equation*}
It is an additive subsemigroup of $\N$: it contains $0$ and is closed under addition.
If $S$ is a numerical semigroup, then $S = \S(A)$ for some $A \in \M_{m(S)}(\Q)$ \cite[Thm.~1.1]{NSRM2}.
This improved upon a result in \cite[Thm.~6.2]{NSRM1}, in which the $g(S) + 1$ replaced $m(S)$.
These results extend to non-numerical semigroups in $\N$ \cite[Cor.~3.6]{NSRM2}.  

The \emph{matricial dimension} $\mdim S$ of a semigroup $S \subseteq \N$ is the smallest $d \geq 1$
such that $S = \S(A)$ for some $A \in \M_d(\Q)$.  Thus, $\mdim S \leq m(S)$ for a numerical semigroup
and, more generally, $\mdim S \leq \min (S \setminus \{0\})$ \cite[Cor.~3.6]{NSRM2}.  One can show, for example, 
that $\mdim \semigroup{6,9,20} = 6$.  The exact computation of the matricial dimension of an arbitrary semigroup in $\N$ remains an open problem.

A semigroup is \emph{dimension-$d$ realizable} if it is the exponent semigroup of a $d \times d$ rational matrix.
One can see that a semigroup is dimension-$d$ realizable if and only if its matricial dimension is at most $d$.
The only dimension-$1$ realizable semigroups are $\{0\}$ and $\N$ \cite[Prop.~2.1]{NSRM1}, so a semigroup is dimension-$2$ realizable if and only if it is trivial or its matricial dimension is $2$.

\subsection*{Main results}
Lucas semigroups, introduced in Section \ref{Section:LucasSemigroup}, are the key to understanding dimension-$2$ realizable semigroups.  Using divisibility properties of Lucas sequences, split along the traditional cases familiar to specialists \cite{BallotPaper, BallotBook}, we describe all such semigroups in Theorems \ref{Theorem:PQZero}, \ref{Theorem:RegPrimeCyclicLocal}, \ref{Theorem:LocalSpecialLucasSGCase1Classifica}, \ref{Theorem:LocalSpecialLucasSGCase2Classifica}, and \ref{Theorem:LocalSpecialLucasSGCase3Classifica}.

Theorem \ref{Theorem:GlobalIff} characterizes the dimension-$2$ realizable semigroups: they are precisely the Lucas semigroups.
Theorem \ref{Theorem:NotGlobal} asserts that there exists a nonlocal Lucas semigroup; that is, a Lucas semigroup whose structure is determined by multiple primes, while
Theorem \ref{Theorem:SmallSet} harnesses the classification of Lucas semigroups to provide a new method for proving a semigroup has matricial dimension greater than $2$.  This is enough to solve an open problem:

The lonely element conjecture asserted that a numerical semigroup whose nonzero small elements are all lonely has matricial dimension at most $2$ (the converse is \cite[Thm.~5.1]{NSRM1}).
Theorem \ref{Theorem:Plus} provides a counterexample the conjecture.

\subsection*{Organization}
This paper is organized as follows. 
Section \ref{Section:LucasSemigroup} introduces Lucas semigroups and their localizations.  
Sections \ref{Section:Local1}, \ref{Section:Local2}, and \ref{Section:Local3} concern the classification of local Lucas semigroups.
The material in Section \ref{Section:Realizable} characterizes dimension-$2$ realizable semigroups:
they are precisely the Lucas semigroups.
Section \ref{Section:Applications} contains several applications and a counterexample to the lonely element conjecture.
We conclude in Section \ref{Section:Open} with several open problems and conjectures.

\subsection*{Acknowledgments}
We thank Christopher O'Neill for many helpful remarks and spotting some important simplifications in an earlier draft of this paper.

\section{Lucas semigroups}\label{Section:LucasSemigroup}
In this section, we define certain additive subsemigroups of $\N = \{0,1,2,\ldots\}$.
related to Lucas sequences.  Although it will take us some time to get there, we will ultimately describe all such Lucas semigroups and prove that a semigroup is dimension-$2$ realizable if and only if it is Lucas  (Theorem \ref{Theorem:GlobalIff}).

\subsection{Lucas sequences}\label{Subsection:Lucas}
The study of semigroups of matricial dimension $2$ leads to the consideration of certain relatives of the Fibonacci numbers.
For $P,Q \in \Z$, define the corresponding (fundamental) \emph{Lucas sequence} 
$U_n = U_n(P,Q)$ by
\begin{equation}\label{eq:Lucas}
U_{n+2} =  P U_{n+1} - Q U_n, \quad U_0 = 0, \quad U_1 = 1.
\end{equation}
If $Q = 0$, then $U_n = P^{n-1}$ for $n \geq 1$.
If $P = 0$, then
\begin{equation}\label{eq:FlipFlop}
U_n = 
\begin{cases}
0 & \text{if $n$ is even},\\
(-Q)^{ \lceil n/2 \rceil } & \text{if $n$ is odd}.
\end{cases}
\end{equation}
The \emph{companion sequence} $V_n = V_n(P,Q)$ is defined by 
$V_{n+2} =  P V_{n+1} - Q V_n$, $V_0 = 2$ and $V_1 = P$.
Both $U_n$ and $V_n$ are integer sequences. 
Let $\alpha,\beta$ denote the zeros of $x^2 - Px + Q$, so $P=\alpha + \beta$ and $Q = \alpha \beta$.
If $\alpha \neq \beta$, then
\begin{equation*}
U_n = \frac{\alpha^n - \beta^n}{\alpha - \beta} 
\quad \text{and} \quad
V_n = \alpha^n+ \beta^n.
\end{equation*}
If $\alpha = \beta$, then they both equal $P/2$, so $U_n = n \alpha^{-1}$ and $V_n  = 2 \alpha^n$ \cite[(2.6)]{BallotBook}.

In what follows, $p$ denotes a prime number, even when this is not explicitly stated.
A Lucas sequence $U_n$ is \emph{regular} if $\gcd(P,Q) = 1$.
A prime $p$ is \emph{regular} if $p \nmid \gcd(P,Q)$; otherwise it is \emph{special}.

Fix $P,Q \in \Z$ and let $R \in \Q$.  The associated \emph{Lucas semigroup} is
\begin{equation}\label{eq:LucasSemigroup}
\LL(P,Q,R) = \{ n \in \N : U_n R \in \Z\},
\end{equation}
in which $U_n = U_n(P,Q)$ is the Lucas sequence \eqref{eq:Lucas}.  
Note that $\LL(P,Q,R) = \N$ for $R \in \Z$.  Our first observation is that Lucas semigroups are semigroups.

\begin{lemma}\label{Lemma:IsSemigroup}
$\LL(P,Q,R)$ is a semigroup for each $P,Q \in \Z$ and $R \in \Q$.
\end{lemma}

\begin{proof}
First $0 \in \LL(P,Q,R)$ since $U_0 = 0$.
Since $U_{m+n} =  U_m U_{n+1} - Q U_{m-1} U_n$ \cite[(2.43)]{BallotBook}, we see that $\LL(P,Q,R)$ is closed under addition.  The associativity of addition ensures that $\LL(P,Q,R)$ is subsemigroup of $\N$.
\end{proof}

\subsection{Local Lucas semigroups}
If $R = p^{-r}$, in which $p$ is prime and $r \in \Z$, then 
$\LL(P,Q,p^{-r})$ is a \emph{local Lucas semigroup}.  Observe that
\begin{equation}\label{eq:LocalDefinition}
\LL(P,Q,p^{-r}) = \{ n \in \N : \nu_p(U_n) \geq r\},
\end{equation}
in which $\nu_p(\cdot)$ denotes the $p$-adic valuation.  Note that $\LL(P,Q,p^{\nu_p(R)}) = \N$ for all but finitely many primes $p$, namely those primes which divide the denominator of $R$ when it is written in lowest terms.   Consequently,
\begin{equation}\label{eq:Intersection}
\LL(P,Q,R) = \bigcap_p \LL(P,Q,p^{\nu_p(R)}),
\end{equation}
in which $p$ runs over the prime numbers. In fact, only the denominator of $R$ matters.
If $R = N/D$, in which $N,D \in \Z$ with $D \neq 0$ and $\gcd(N,D) = 1$, then
\begin{align}
    \LL(P,Q,R) 
    &= \bigcap_p \LL(P,Q,p^{\nu_p(R)}) 
    = \bigcap_{p\mid D} \LL(P,Q,p^{\nu_p(R)}) \nonumber \\
    &=\bigcap_{p} \LL(P,Q,p^{-\nu_p(D)})
    =\LL(P,Q,\tfrac{1}{D}). \label{eq:OnlyDemon}
\end{align} 

The next lemma asserts that the intersection of Lucas semigroups, for the same sequence parameters $P$ and $Q$,
is a Lucas semigroup for those same parameters.

\begin{lemma}\label{Lemma:FixedPQClosure}
Fix $P,Q \in \Z$.
If $R_1,R_2,\ldots,R_n \in \Q$, then
$\bigcap_{i=1}^n \LL(P,Q,R_i) = \LL(P,Q,R)$ for some $R \in \Q$.
\end{lemma}

\begin{proof}
For $i=1,2,\ldots,n$, write $R_i = N_i / D_i$, where $N_i,D_i \in \Z$ with $D_i \neq 0$ and $\gcd(N_i,D_i)=1$.
Let $D = \prod_{p \mid D_1 D_2 \cdots D_n} p^{\max\{ \nu_p(D_1), \nu_p(D_2),\ldots, \nu_p(D_n) \}}$ and $R =1/D$, so
\begin{align*}
\bigcap_{i=1}^n \LL(P,Q,R_i)
&= \bigcap_{i=1}^n \LL(P,Q,1/D_i)
=\bigcap_{i=1}^n \bigcap_{p } \LL(P,Q,p^{-\nu_p(D_i)}) \\
&=  \bigcap_{p }\bigcap_{i=1}^n \LL(P,Q,p^{-\nu_p(D_i)}) 
=  \bigcap_{p } \LL\big(P,Q,p^{-\max\{ \nu_p(D_1), \ldots, \nu_p(D_n) \} } \big) \\
&= \LL(P,Q,1/D) = \LL(P,Q,R).\qedhere
\end{align*}
\end{proof}

The following result implies that the small elements of a Lucas semigroup must be lonely:
if two consecutive natural numbers belong to a Lucas semigroup, then the semigroup
contains all subsequent natural numbers.

\begin{proposition}\label{Proposition:TwoTerm}
If $n,n+1 \in \LL(P,Q,R)$, then $\{n,n+1,\ldots\} \subseteq \LL(P,Q,R)$.
\end{proposition}

\begin{proof}
Suppose that $p^r$ exactly divides the denominator of $R$ (when written in lowest terms).
If $n,n+1 \in \LL(P,Q,R)$, then 
\begin{align*}
\nu_p(U_{n+2}) 
&= \nu_p( P U_{n+1} - Q U_n) 
\geq \min\{ \nu_p(P U_{n+1}), \nu_p(Q U_n)\} \\
&\geq \min\{ \nu_p(P) + \nu_p(U_{n+1}), \nu_p(Q) + \nu_p(U_n)\} \\
&\geq \min\{  \nu_p(U_{n+1}),\nu_p(U_n)\} \geq r,
\end{align*}
so $n+2 \in \LL(P,Q,p^{-r})$.  Induction ensures $\{n,n+1,\ldots\} \subseteq \LL(P,Q,p^{-r})$
and \eqref{eq:Intersection} implies that $\{n,n+1,\ldots\} \subseteq \LL(P,Q,R)$.
\end{proof}

To understand the structure of Lucas semigroups we need only study local Lucas semigroups; see \eqref{eq:Intersection}.  
In what follows, we consider three cases:
\begin{enumerate}
\item $PQ = 0$ (Section \ref{Section:Local1}).
\item $PQ \neq 0$ and $p \nmid \gcd(P,Q)$ (Section \ref{Section:Local2}).
\item $PQ \neq 0$ and $p \mid \gcd(P,Q)$ (Section \ref{Section:Local3}).
\end{enumerate}

\section{Local Lucas semigroups, Case I: $PQ = 0$}\label{Section:Local1}
Fix $P,Q \in \Z$ and let $U_n = U_n(P,Q)$ denote the Lucas sequence \eqref{eq:Lucas}.  
Local Lucas semigroups \eqref{eq:LocalDefinition} with $PQ = 0$ are the simplest of the three cases we must consider.
We summarize our findings in this case in the following theorem (recall from \eqref{eq:TailDefinition} the definition
$S_m = \{0 , m,m+1,\ldots\}$).

\begin{theorem}\label{Theorem:PQZero}
If $PQ = 0$ and $r \geq 1$,
\begin{equation*}
\LL(P,Q,p^{-r}) = 
\begin{cases}
\{0\} & \text{if $P \neq 0$, $Q = 0$, and $p \nmid P$},\\[3pt]
S_{ \lceil r / \nu_p(P)\rceil + 1 } & \text{if $P \neq 0$, $Q = 0$, and $p \mid P$},\\[3pt]
\semigroup{2} & \text{if $P=0$, $Q \neq 0$, and $p \nmid Q$},\\[4pt]
\semigroup{2, \,2\big\lceil r / \nu_p(Q) \big\rceil \!+\!1} & \text{if $P=0$, $Q \neq 0$, and $p \mid Q$},\\[4pt]
S_2 & \text{if $P = Q = 0$}.
\end{cases}
\end{equation*}
\end{theorem}

\begin{proof}
\noindent\textbf{Case 1.}
Let $P \neq 0$, $Q=0$, and $p \nmid P$.
Since $U_n = P^{n-1}$ for $n \geq 1$, we have $\nu_p(U_n) = 0$ for $n \geq 0$.
Thus, $\LL(P,Q,p^{-r}) = \{0\}$.  

\smallskip\noindent\textbf{Case 2.} 
Let $P \neq 0$, $Q = 0$, and $p \mid P$. Since $U_n = P^{n-1}$ for $n \geq 1$, we have
$\nu_p(U_n) \geq r$ if and only if
$(n-1) \nu_p(P) \geq r$ if and only if 
$n \geq \lceil r / \nu_p(P) \rceil + 1$,
so $\LL(P,Q,p^{-r}) =S_{ \lceil r / \nu_p(P) \rceil + 1 }$.

\smallskip\noindent\textbf{Case 3.} 
Let $P=0$, $Q \neq 0$, and $p \nmid Q$.  Then 
$U_{2i}=0$ and $U_{2i+1} = (-Q)^i$ \cite[p.~13]{BallotBook},
so $\LL(P,Q,p^{-r}) = \semigroup{2}$.

\smallskip\noindent\textbf{Case 4.} 
Let $P=0$, $Q \neq 0$, and $p \nmid Q$.  As in Case 3, we have $\semigroup{2} \subseteq  \LL(P,Q,p^{-r})$
because each $U_{2i} = 0$.
Since $p \mid Q$, we have
$\nu_p(U_{2i+1}) \geq r$ if and only if
$i \nu_p(Q) \geq r$
if and only if
$i \geq \lceil r / \nu_p(Q) \rceil$,
which means $\LL(P,Q,p^{-r}) = \semigroup{2, \,2\lceil r / \nu_p(Q) \rceil +1}$.

\smallskip\noindent\textbf{Case 5.} 
If $P = Q = 0$, then $U_n = 0$ for $n\geq 2$, so $\LL(P,Q,p^{-r}) = \N \backslash \{1\}$.
\end{proof}

The previous theorem permits us to characterize Lucas semigroups with $PQ = 0$.
From \eqref{eq:Intersection}, we see that each such semigroup is an intersection of local Lucas semigroups 
described by Theorem \ref{Theorem:PQZero}. Recall the definition \eqref{eq:ModularTail} of $S_m(n)$.

\begin{corollary}\label{Corollary:PQZero}
If $PQ = 0$ and $R \in \Q \setminus \Z$, then $\LL(P,Q,R)$ is one of the following:
\begin{enumerate}[leftmargin=*]
\item $\{0\}$,
\item $S_m = \{0,m,m+1,m+2,\ldots\}$ for some $m \geq 2$,
\item $\semigroup{2}$, 
\item $\semigroup{2,m}$ for some odd $m \geq 3$.
\end{enumerate}
Each such semigroup occurs in this manner and each is dimension-$2$ realizable.
\end{corollary}

\begin{proof}
Consider intersections of the semigroups described in the previous theorem and use 
the fact that $S_a(g_1)\cap S_b(g_2)=S_{\max(a,b)}(\lcm\{g_1,g_2\})$.
The data in Table \ref{Table:Realize1} show that the semigroups listed above are dimension-$2$ realizable.
\end{proof}

\begin{table}
    \begin{equation*}
   \begin{array}{c|cc}
    S  & A & \mdim S \\[2pt]
    \hline
    \{0\} & \tinymatrix{ \frac{1}{2} }{0}{0}{0}  & 1 \\[8pt] 
     \text{$S_k$ with $k\geq 2$} & \tinymatrix{2}{2^{1-k}}{0}{0} & 2 \\[8pt]
    \semigroup{2} & \tinymatrix{1}{ \frac{1}{2} }{0}{1} & 2 \\[8pt]
    \text{$\semigroup{2,k}$ with $k \geq 3$ odd} & \scriptsize\minimatrix{0}{2^{-\frac{k-1}{2}}}{2^{\frac{k+1}{2}}}{0}& 2 
    \end{array}
    \end{equation*}
    \caption{The semigroups in Corollary \ref{Corollary:PQZero} are dimension-$2$ realizable.  Here $S = \S(A)$ for the indicated matrix.
    The second entry above corrects an off-by-one error in \cite[Table 1]{NSRM1}.}
    \label{Table:Realize1}
\end{table}

\section{Local Lucas semigroups, Case II: $PQ \neq 0$ and $p$ is regular}\label{Section:Local2}

Recall that a prime $p$ is \emph{regular} for $U_n = U_n(P,Q)$ if $p \nmid \gcd(P,Q)$.
The \emph{rank of appearance} of $m \geq 2$ in a Lucas sequence $U_n$
is the least $\rho = \rho(m)\geq 2$, if it exists, such that $m \mid U_{\rho}$ \cite[Sec.~2.4, Def.~1]{BallotBook}.
If $\gcd(m,Q) = 1$, then $\rho(m)$ exists \cite[Sec.~2.4, Lem.~8]{BallotBook}.  In particular,
if $p$ is a regular prime, then $\rho(p)$ exists.  Since $p$ remains fixed throughout much of the following, we let
$\rho = \rho(p)$.  The \emph{rank exponent} of $p$ is $\nu = \nu_p(U_{\rho})$.
In what follows, $(x)_+=\max\{x,0\}$.

\begin{theorem}\label{Theorem:RegPrimeCyclicLocal}
If $PQ \neq 0$, $p \nmid \gcd(P,Q)$, and $r \geq 1$, then
\begin{equation*}
\LL(P,Q,p^{-r})= \begin{cases}
 \{0\} & \text{if $p \mid Q$}, \\[3pt]
 \semigroup{ p^{ (r-\nu)_+}\rho} & \text{if $p \geq 3$ and $p \nmid Q$},\\[3pt]
 \semigroup{  2^{ 1 + (r - \nu_2(P))_+ }  } & \text{if $p= 2 \mid P$ and $p \nmid Q$},\\[3pt]
 \semigroup{ 3 \cdot 2^{(r - \nu_2(U_3))_+} }& \text{if $p=2 \nmid P$ and $p \nmid Q\equiv 1\pmod{4}$},\\[3pt]
\semigroup{ 3 } & \text{if $p=2 \nmid P$, $p \nmid Q\equiv -1\pmod{4}$, and $r =1$},\\[3pt]
 \semigroup{ 6 \cdot 2^{(r - \nu_2(U_6) )_+ } } &\text{if $p=2 \nmid P$, $p \nmid Q\equiv -1\pmod{4}$, and $r\neq1$}.
 \end{cases}
\end{equation*}
\end{theorem}

\begin{proof}
Recall from \eqref{eq:LocalDefinition} that $n \in \LL(P,Q,p^{-r}) \iff \nu_p(U_n)  \geq r$.

\smallskip\noindent\textbf{Case 1.}
Suppose $p \mid Q$. Then $p \nmid U_n$ for $n \geq 1$ by \cite[Sect.~2.4, Lem.~6]{BallotBook},
so each $\nu_p(U_n) = 0$.  Thus, $L(P,Q,p^{-r}) = \{0\}$.

\smallskip\noindent\textbf{Case 2.} Suppose $p \geq 3$ and $p \nmid Q$. Then $\LL(P,Q,p^r) \subseteq \semigroup{\rho}$ since \cite{BallotErrata} ensures that
\begin{equation*}
\nu_p(U_n) =
\begin{cases} 
0 & \text{if $\rho \nmid n$},\\
\nu + \nu_p\big(\frac{n}{\rho}\big), & \text{if $\rho \mid n$}.
\end{cases}
\end{equation*}
Then $\LL(P,Q,2^{-r}) = \semigroup{ p^{ (r-\nu)_+}\rho}$, so
\begin{equation*}
\nu_p(U_{\rho i}) \geq r 
\iff \nu + \nu_p(i) \geq r  
\iff \nu_p( i) \geq  (r-\nu)_+.
\end{equation*}

\smallskip\noindent\textbf{Case 3.}
Suppose $p=2 \mid P$ and $p \nmid Q$. Then $P$ is even and $Q$ is odd, so 
\begin{equation*}
\nu_2(U_{2i}) = \nu_2(iP)
\quad\text{and}\quad 
\nu_2(U_{2i+1}) = 0 
\end{equation*}
for $i \geq 0$ \cite[Sec.~2.5, Thm.~13]{BallotBook}.  Thus, $\LL(P,Q,p^{-r}) \subseteq \semigroup{2}$ and
\begin{align*}
\nu_2(U_{2i}) \geq r
&\iff \nu_2(iP) \geq r 
\iff \nu_2(i) + \nu_2(P) \geq r \\
&\iff \nu_2(i) \geq ( r-\nu_2(P) )_+ \\
&\iff \nu_2(2i) \geq 1 + (r - \nu_2(P))_+,
\end{align*}
so $\LL(P,Q,2^{-r}) = \semigroup{ 2^{ 1 + (r - \nu_2(P))_+ } }$.

\smallskip\noindent\textbf{Case 4.}
Suppose $p=2 \nmid P$ and $p \nmid Q$.  Then $P,Q$ are odd, so
$U_n$ is even if and only if $3 \mid n$ \cite[IV.18 page 60]{Ribenboim1996TheNB}.
Thus, $\LL(P,Q,p^{-r}) \subseteq \semigroup{3}$.
\begin{enumerate}[leftmargin=*]
\item If $Q \equiv 1 \pmod{4}$, then $4 \mid U_3$ and  $\nu_2(U_{3n}) = \nu_2(nU_3)$ for $n \geq 0$ \cite[Thm.~14, Sec.~2.5]{BallotBook}.  
Then $\LL(P,Q,p^{-r}) = \semigroup{ 3 \cdot 2^{r - \nu_2(U_3)} }$ since
\begin{align*}
\nu_2(U_{3n}) \geq r 
&\iff \nu_2(nU_3) \geq r  \\
&\iff \nu_2(n) \geq r - \nu_2(U_3) \\
&\iff 3n \in \langle 3 \cdot 2^{r - \nu_2(U_3)} \rangle.
\end{align*}

\item If $Q \equiv -1 \pmod{4}$, then $2 \parallel U_{6n+3}$ (the symbol $\parallel$ means ``exactly divides''), $\nu_2(U_6) \geq 3$, and 
$ \nu_2(U_{6n}) = \nu_2(nU_6) = \nu_2(n) + \nu_2(U_6) \geq \nu_2(n) + 3$ \cite[Thm.~14, Sec.~2.5]{BallotBook},
so $\LL(P,Q,p^{-r}) \subseteq \semigroup{3}$ as $\nu_2(U_n)\geq 1 \iff 3\mid n$.

If $r=1$, then since $2 \parallel U_{6n+3}$ we know that $\nu_2(U_{6n+3})= 1=r$ for every $n \in \N$ and by the discussion above $\nu_2(U_{6n})\geq 3 \geq 1=r$, so $\nu_2(U_{3n})\geq 1=r$ for every $n \in \N$. Hence, $\LL(p,q,p^{-r})=\semigroup{3}$.

Now assume $r \neq 1$. Then $r \geq 2$. Since $2 \parallel U_{6n+3}$ we know that $\nu_2(U_{6n+3})= 1 < r$. So, there does not exist an $n \in \N$ such that $\nu_2(U_{6n+3}) \geq r$. Therefore, $\LL(P,Q,p^{-r})\subseteq \semigroup{6}$ for $r\geq 2$.

Moreover,
\begin{align*}
\nu_2(U_{6n}) \geq r 
&\iff \nu_2(n) + \nu_2(U_6) \geq r \\
&\iff \nu_2(n) \geq r - \nu_2(U_6) \\
&\iff 6n \in \semigroup{ 6 \cdot 2^{r - \nu_2(U_6)} },
\end{align*}
so $\LL(P,Q,p^{-r}) =  \semigroup{ 6 \cdot 2^{r - \nu_2(U_6)} }$. 
 \qedhere
\end{enumerate}
\end{proof}

\begin{table}\footnotesize
\begin{equation*}
\begin{array}{c|cccc|l}
m & P & Q & p & r & \hspace{1.5in} U_n(P,Q)   \\[2pt]
\hline
2 & 3 & 5 & 3 & 1 & \{0, 1, 3, 4, -3, -29, -72, -71, 147, 796, 1653, 979,\ldots\} \\[2pt]
3 & 3 & 2 & 7 & 1 & \{0,1,3,7,15,31,63,127,255,511,1023,2047,\ldots\}\\[2pt]
4 & 1 & 2 & 3 & 1 & \{0,1,1,-1,-3,-1,5,7,-3,-17,-11,23,\ldots\} \\[2pt]
5 & 3 & 1 & 5 & 1 & \{0, 1, 3, 8, 21, 55, 144, 377, 987, 2584, 6765, 17711,\ldots\} \\[2pt]
6 & 1 & 2 & 5 & 1 & \{0, 1, 1, -1, -3, -1, 5, 7, -3, -17, -11, 23, 45, -1, -91, -89,\ldots\} \\[2pt]
7 & 1 & 2 & 7 & 1 & \{0, 1, 1, -1, -3, -1, 5, 7, -3, -17, -11, 23, 45, -1, -91, -89, 93, 271,\ldots\} \\[2pt]
8 & 4 & 3 & 41 & 1 & \{0, 1, 4, 13, 40, 121, 364, 1093, 3280, 9841, 29524, 885730,\ldots\} \\[2pt]
9 & 1 & 2 & 17 & 1 & \{0, 1, 1, -1, -3, -1, 5, 7, -3, -17, -11, 23, 45, -1, -91, -89, 93,\ldots\} \\[2pt]
10 & 1 & 2 & 11 & 1 & \{0, 1, 1, -1, -3, -1, 5, 7, -3, -17, -11, 23, 45, -1, -91, -89, 93,\ldots\} \\[2pt]
11 & 1 & 2 & 23 & 1 & \{0, 1, 1, -1, -3, -1, 5, 7, -3, -17, -11, 23, 45, -1, -91, -89, 93,\ldots\} \\[2pt]
12 & 3 & 1 & 23 & 1 & \{0, 1, 3, 8, 21, 55, 144, 377, 987, 2584, 6765, 17711, 46368, 121393,\ldots\} \\[2pt]
13 & 3 & 2 & 8191 & 1 & \{0, 1, 3, 7, 15, 31, 63, 127, 255, 511, 1023, 2047, 4095, 8191, 16383,\ldots\} \\[2pt]
14 & 1 & 2 & 13 & 1 & \{0, 1, 1, -1, -3, -1, 5, 7, -3, -17, -11, 23, 45, -1, -91, -89,\ldots\} \\[2pt]
15 & 3 & 1 & 31 & 1 & \{0, 1, 3, 8, 21, 55, 144, 377, 987, 2584, 6765, 17711, 46368, 121393,317811,\ldots\} \\[2pt]
16 & 1 & 2 & 31 & 1 & \{0, 1, 1, -1, -3, -1, 5, 7, -3, -17, -11, 23, 45, -1, -91, -89,\ldots\} \\[2pt]
17 & 4 & 1 & 67 & 1 & \{0, 1, 4, 15, 56, 209, 780, 2911, 10864, 40545, 151316, 564719, 2107560,\ldots\} \\[2pt]
18 & 3 & 1 & 107 & 1 & \{0, 1, 3, 8, 21, 55, 144, 377, 987, 2584, 6765, 17711, 46368, 121393,317811,\ldots\} \\[2pt]
19 & 3 & 1 & 37 & 1 & \{0, 1, 3, 8, 21, 55, 144, 377, 987, 2584, 6765, 17711, 46368, 121393,317811,\ldots\} \\[2pt]
20 & 1 & 2 & 19 & 1 & \{0, 1, 1, -1, -3, -1, 5, 7, -3, -17, -11, 23, 45, -1, -91, -89, 93,271, 85,\ldots\} \\[2pt]
21 & 3 & 2 & 337 & 1 & \{0, 1, 3, 7, 15, 31, 63, 127, 255, 511, 1023, 2047, 4095, 8191, 16383, 32767,\ldots\} \\[2pt]
22 & 3 & 1 & 43 & 1 & \{0, 1, 3, 8, 21, 55, 144, 377, 987, 2584, 6765, 17711, 46368, 121393,317811\ldots\} \\[2pt]
23 & 13 & 1 & 47 & 1 & \{0, 1, 13, 168, 2171, 28055, 362544, 4685017, 60542677, 782369784,\ldots\} \\[2pt]
24 & 11 & 1 & 47 & 1 & \{0, 1, 11, 120, 1309, 14279, 155760, 1699081, 18534131, 202176360,\ldots\} \\[2pt]
25 & 3 & 2 & 1801 & 1 & \{0, 1, 3, 7, 15, 31, 63, 127, 255, 511, 1023, 2047, 4095, 8191, 16383,32767,\ldots\} \\[2pt]
26 & 3 & 2 & 2731 & 1 & \{0, 1, 3, 7, 15, 31, 63, 127, 255, 511, 1023, 2047, 4095, 8191, 16383, 32767,\ldots\} \\[2pt]
27 & 3 & 1 & 53 & 1 & \{0, 1, 3, 8, 21, 55, 144, 377, 987, 2584, 6765, 17711, 46368, 121393,317811,\ldots\} \\[2pt]
28 & 1 & 2 & 29 & 1 & \{0, 1, 1, -1, -3, -1, 5, 7, -3, -17, -11, 23, 45, -1, -91, -89, 93,271, 85,\ldots\} \\[2pt]
29 & 3 & 1 & 59 & 1 & \{0, 1, 3, 8, 21, 55, 144, 377, 987, 2584, 6765, 17711, 46368, 121393,317811,\ldots\} \\[2pt]
30 & 3 & 2 & 331 & 1 & \{0, 1, 3, 7, 15, 31, 63, 127, 255, 511, 1023, 2047, 4095, 8191, 16383, 32767,\ldots\} 
\end{array}
\end{equation*}
\caption{Data for the proof of Corollary \ref{Corollary:Regular}, Case 1: parameters such that $\LL(P,Q,p^{-r}) = \semigroup{m}$ for $2 \leq m \leq 30$.}
\label{Table:Thirty}
\end{table}

Every cyclic semigroup in $\N$, except $\N$ itself, occurs in the following result.

\begin{corollary}\label{Corollary:Regular}
Let $PQ \neq 0$.  If each prime that divides the denominator of $R \in \Q$ is regular, then $\LL(P,Q,R)$ is 
$\{0\}$ or $\semigroup{m}$ for some $m \geq 2$.
Each such semigroup occurs in this manner and each is dimension-$2$ realizable.
\end{corollary}

\begin{proof}
The intersection of cyclic semigroups is cyclic, so the first statement follows from Theorem \ref{Theorem:RegPrimeCyclicLocal}.
The first case in Theorem \ref{Theorem:RegPrimeCyclicLocal} ensures that $\{0\}$ occurs,
and the remaining four cases ensure that $\N$ never occurs (remember that $\rho \geq 2$).
Since $\S([\frac{1}{2}]) = \{0\}$ and $\S(\tinymatrix{1}{1/m}{0}{1}) = \semigroup{m}$ for $m \geq 2$, 
it suffices to show that $\semigroup{m}$ occurs in Theorem \ref{Theorem:RegPrimeCyclicLocal} for each $m \geq 2$.

\smallskip\noindent\textbf{Case 1.} For $2 \leq m \leq 30$, Table \ref{Table:Thirty} provides 
parameters $P$ and $Q$, a prime $p \geq 3$, and exponent $r \geq 1$ such that $\LL(P,Q,p^{-r}) = \semigroup{m}$.

\smallskip\noindent\textbf{Case 2.} Suppose $m \geq 31$.  Each such $m$ is 
\emph{totally nondefective} \cite[Thm.~1.4]{Bilu}:  the $m$th term $U_m$ in every
Lucas sequence $U_n = U_n(P,Q)$ with nonzero \emph{discriminant} $D=P^2-4Q$
has a \emph{primitive} divisor, that is, a prime $p$ such that $p \mid U_m$ but $p \nmid D U_1 U_2 \cdots U_{m-1}$.
Let $P=1$ and $Q=2$, so $D = -7 \neq 0$ and $U_n = U_n(P,Q)$ is 
\begin{equation*}
0, 1, 1, -1, -3, -1, 5, 7, -3, -17, -11, 23, 45, -1, -91, -89, 93, 271, 85, -457,\ldots.
\end{equation*}
Each $U_n$ is odd for $n \geq 1$ \cite[Lem.~1, Sec.~2.3]{BallotBook}.
If $p$ is a primitive divisor of $U_m$ (that is, $m$ is the rank of appearance of $p$), then $p \geq 3$ and $p \nmid Q$.
Let $r = \nu = \nu_p(U_m)$.
The second case of Theorem \ref{Theorem:RegPrimeCyclicLocal} ensures that 
$\LL(P,Q,p^{-r}) = \semigroup{m}$, as desired. See Table \ref{Table:Cyclic} for several illustrative examples.
\end{proof}

\begin{table}\small
\begin{equation*}
\begin{array}{c|ccc|c|ccc}
m & U_m & p & r & m & U_m & p & r\\
\hline
31 & -7193 & 7193 & 1 & 41 & -703889 & 409 ,\,  1721 & 1 ,\,  1 \\
32 & 41757 & 449 & 1 & 42 & 802165 & 43 & 1 \\
33 & 56143  & 2441 & 1 & 43 & 2209943 & 257 ,\,  8599 & 1 ,\,  1 \\
34 & -27371 & 101 & 1 & 44 & 605613 &  131 & 1 \\
35 & -139657 & 71 ,\,  281 & 1 ,\,  1 & 45 &-3814273 & 2521 & 1 \\
36 & -84915 & 37 & 1 & 46 &  -5025499 & 5197 & 1 \\
37 & 194399 & 73 ,\,  2663 & 1 & 47 & 2603047 & 2603047 & 1 \\
38 & 364229 & 797 & 1 & 48 & 12654045 & 193 & 1\\
39 & -24569 & 79 ,\,  311 & 1 ,\,  1 & 49 & 7447951 & 97 ,\,  1567 & 1 ,\,  1 \\
40 & -753027 & 1201 & 1 & 50 & -17860139 & 401 & 1\\
\end{array}
\end{equation*}
\caption{Parameters that lead to $\LL(P,Q,p^{-r}) = \semigroup{m}$ for $31 \leq m \leq 50$,
in which $P = 1$ and $Q=2$ and $U_n =U_n(P,Q)$.}
\label{Table:Cyclic}
\end{table}

\section{Local Lucas semigroups, Case III: $PQ \neq 0$ and $p$ is special}\label{Section:Local3}
Recall that $p$ is \emph{special} if $p \mid \gcd(P,Q)$; that is, if $p \mid P$ and $p \mid Q$.
This is the most complicated of the three main cases, and we must break things down into three cases
to address the complexity.  First, we make an important observation.

\begin{theorem}\label{Theorem:SpecialNumerical}
If $p$ is a special prime, then $\LL(P,Q,p^{-r})$ is a numerical semigroup.
That is, the semigroup $\LL(P,Q,p^{-r})$ has finite complement in $\N$.
\end{theorem}

\begin{proof}
Since \cite[Sec.~2.12, Thm.~42]{BallotBook} ensures that $\nu_p(U_n) \geq \lfloor \frac{n}{2} \rfloor$ for $n \geq 1$, it follows that
$\{2r,2r+1,\ldots\}\subseteq \LL(P,Q,p^{-r})$.  Thus, $\N \setminus \LL(P,Q,p^{-r})$ is finite.
\end{proof}

Suppose $p$ is a special prime for $U_n$, in which $PQ \neq 0$.
Then $p \mid \gcd(P,Q)$ and we may write $P = p^a P'$ and $Q = p^b Q'$, in which $a,b \geq 1$ and $p \nmid P'Q'$.
We adhere to this notation throughout this section without further comment.

We consider three cases: $b > 2a$, $b=2a$, and $b < 2a$. 

\subsection{Special primes: Case 1}\label{Subsection:Special1} The first case, in which $b > 2a$, is straightforward: every Lucas semigroup in this case is an ordinary semigroup.

\begin{theorem}\label{Theorem:LocalSpecialLucasSGCase1Classifica}
    If $p$ is special, $b > 2a$, and $r\geq 1$, then
    $\LL(P,Q,p^{-r})=S_{\ceiling{\frac{r+a}{a}}}$.
\end{theorem}

\begin{proof}
If $b > 2a$, then $\nu_p(U_n) = (n-1)a$ \cite[Thm.~1.2]{BallotPaper}.
Since
\begin{equation}\label{eq:SpecialChain}
\nu_p(U_n) \geq r
\iff (n-1)a \geq r
\iff n-1\geq \Big\lceil \frac{r}{a} \Big\rceil \iff n\geq \Big\lceil \frac{r+a}{a} \Big\rceil  ,
\end{equation}
we conclude that $\LL(P,Q,p^{-r}) = S_{\lceil \frac{r+a}{a} \rceil  }$.
\end{proof}

\subsection{Special primes: Case 2}\label{Subsection:Special2}

The second case, in which $b=2a$, is more complicated than the previous case.  
Maintain the notation of the previous subsection and let $\rho' = \rho'(p) \geq 2$ denote 
the rank of appearance of $p$ in $U'_n=U_n'(P',Q')$ and let $\nu' = \nu_p(U'_{\rho'})$ denote
the rank exponent of $p$ in $U_n'$; see Subsection \ref{Section:Local2}.  Recall that $\s(S)$ denotes the set of nonzero
small elements of $S$. Recall that $PQ \neq 0$.

\begin{theorem}\label{Theorem:LocalSpecialLucasSGCase2Classifica}
    If $p\geq 3$ is special, $b=2a$, and $r\geq 1$, then
    \begin{equation*}
    \LL(P,Q,p^{-r})=S_{\ceiling{\frac{r+a}{a}}}\cup\{n\in \semigroup{\rho'}:\nu_p\big(\tfrac{n}{\rho'}\big)\geq r+a-an-\nu'\}.
    \end{equation*}
    If $\rho'\geq \ceiling{\frac{r+a}{a}}$, then 
    $\LL(P,Q,p^{-r})=S_{\ceiling{\frac{r+a}{a}}}$.
\end{theorem}

\begin{proof}
If $b = 2a$, then
\begin{equation}\label{eq:TooSpecial}
\nu_p(U_n)=(n-1)a+\nu_p(U_n'),
\end{equation}
in which $U_n' = U_n'(P',Q')$ \cite[Thm.~1.2]{BallotPaper}.  
Since $p \nmid P'Q'$ and $P'Q'\neq 0$, it follows that $p$ is a regular prime for $U_n'$, so $\rho'$ and $\nu'$ are well defined.  Thus,
\begin{align}
    \LL(P,Q,p^{-r})
    &=\{n\in \N: \nu_p(U_n)\geq r\}  \nonumber \\
    &=\{n\in \N:(n-1)a+\nu_p(U_n') \geq r\} \nonumber \\
    &=\{n \in \N:\nu_p(U_n')\geq r+a-an\} \label{eq:FromHere}\\
        &=S_{\ceiling{\frac{r+a}{a}}}\cup \{n\in \semigroup{ \rho' }:\nu'+\nu_p\big(\tfrac{n}{\rho'}\big)\geq r+a-an\} \nonumber \\
    &=S_{\ceiling{\frac{r+a}{a}}}\cup\{n\in \semigroup{\rho' }:\nu_p\big(\tfrac{n}{\rho'}\big)\geq r+a-an-\nu'\} .\nonumber 
    \end{align}
In light of \eqref{eq:TooSpecial}, 
it follows that $g = g(\LL(P,Q,p^{-r}))\leq \big\lceil \frac{r}{a} \big\rceil$ since
\begin{equation*}
(n-1)a\geq r \iff n-1\geq \Big\lceil \frac{r}{a} \Big\rceil \iff n\geq \Big\lceil \frac{r}{a} \Big\rceil  + 1.
\end{equation*}
Because $p \nmid Q'$, we know that $\nu_p(U_n') >0$ if and only if $\rho \mid n$ \cite{BallotErrata}.
Since $p \nmid U_1=1$, we have $\rho(p)\geq 2$.
If $g = \left\lceil\frac{r}{a}\right\rceil$, we are done. 
Otherwise, $\left\lceil\frac{r}{a}\right\rceil \in \LL(P,Q,p^{-r})$ and hence
$g = \lceil\frac{r}{a}\rceil-1$ is the Frobenius number.    
    
If $\rho'\geq \ceiling{\frac{r+a}{a}}$, then $\rho'\N \subseteq S_{\ceiling{\frac{r+a}{a}}}$. Thus, $\LL(P,Q,p^{-r})=S_{\ceiling{\frac{r+a}{a}}}$.
\end{proof}

The next result says that 
the corresponding local Lucas semigroup is a finite union of scaled ordinary semigroups
$S_m(n) = S_m \cap \semigroup{n}$.  


\begin{corollary}\label{Corollary:Special2H}
If $p$ is special, $b=2a$, and $r \geq 1$, then
\begin{equation}\label{eq:LLC}
\LL(P,Q,p^{-r})=S_{ \ceiling{ \frac{r+a}{a}}} \cup \bigcup_{i=1}^{\ceiling{\frac{r+a}{a}}}S_i(h_i)
\end{equation}
is a numerical semigroup, in which
\begin{equation*}
h_i= \begin{cases}
 { \rho' p^{ (r+a-ai-\nu')_+}} & \text{if $p \geq 3$},\\[3pt]
 { 3 \cdot 2^{(r+a-ai - \nu_2(U'_3))_+} }& \text{if $p=2$ and $Q'\equiv 1\pmod{4}$},\\[3pt]
{ 3 } & \text{if $p=2$, $Q'\equiv -1\pmod{4}$, and $r+a-ai =1$},\\[3pt]
 { 6 \cdot 2^{(r
 +a-ai - \nu_2(U'_6) )_+ } } &\text{if $p=2$ and $Q'\equiv -1\pmod{4}$, and $r+a-ai\neq1$}.
 \end{cases}
\end{equation*}
Moreover, $\s(S) = \varnothing$ or $\gcd \s(S) \geq 2$.
\end{corollary}

\begin{proof}
For $\ell \geq 1$, Theorem~\ref{Theorem:RegPrimeCyclicLocal} ensures that
\begin{equation*}
\LL(P',Q',p^{-\ell})= \begin{cases}
 \semigroup{ p^{ (\ell-\nu')_+}\rho'} & \text{if $p \geq 3$},\\[3pt]
 \semigroup{ 3 \cdot 2^{(\ell - \nu_2(U'_3))_+} }& \text{if $p=2$ and $Q'\equiv 1\pmod{4}$},\\[3pt]
\semigroup{ 3 } & \text{if $p=2$, $Q'\equiv -1\pmod{4}$, and $\ell =1$},\\[3pt]
 \semigroup{ 6 \cdot 2^{(\ell - \nu_2(U'_6) )_+ } } &\text{if $p=2$ and $Q'\equiv -1\pmod{4}$, and $\ell\neq1$}.
 \end{cases}
\end{equation*}
Let $\LL(P',Q',p^{-(r+a-ai)})=\semigroup{h_i}$. 
For $i \in \Z$,
\begin{equation*}
r+a-ai \geq 0
\iff \frac{r+a}{a} \geq i
\iff \Big\lceil \frac{r+a}{a} \Big\rceil \geq i.
\end{equation*}
Resume from \eqref{eq:FromHere} and get
\begin{align*}
    \LL(P,Q,p^{-r})
    &=\{n \in \N:\nu_p(U_n')\geq r+a-an\} \\
    &= \bigcup_{i=0}^{ \ceiling{ \frac{r+a}{a} } } \{n \in \N : \text{$n \geq i$ and  $\nu_p(U_n')\geq r+a-ai$} \}  \\
    &= \bigcup_{i=0}^{ \ceiling{ \frac{r+a}{a} } } \LL(P',Q',p^{-(r+a-ai)}) \cap S_i
    = \bigcup_{i=0}^{ \ceiling{ \frac{r+a}{a} } } \semigroup{h_i}  \cap S_i  \\
    &=\bigcup_{i=0}^{ \ceiling{ \frac{r+a}{a} } } S_i(h_i).  
\end{align*}
Suppose that $\s(S) \neq \varnothing$.
Then every small element of $S$ is a multiple of some $h_i$.  
If $p\geq 3$, then $2 \leq \rho' \mid \gcd \s(S)$.
If $p = 2$, then $3 \mid \gcd \s(S)$.
\end{proof}

\begin{remark}\label{Remark:DecompSmalElemsSpPCase2}
If $h_i\geq \ceiling{\frac{r}{a}}+1$, then $h_i$ exceeds the Frobenius number of \eqref{eq:LLC}.  
\end{remark}

\subsection{Special primes: Case 3}\label{Subsection:Special3}
The third and final case in the special-primes setting is $b < 2a$.
Recall that $U = U_n(P,Q)$, in which $PQ \neq 0$ and $p \mid \gcd(P,Q)$;
that is, $p$ is special.  We write $P=p^aP'$ and $Q=p^bQ'$, in which $p \nmid P'Q'$.  

Since $b < 2a$ and $n \geq 0$, \cite[Thm.~1.2]{BallotPaper} ensures that
$\nu_p(U_{2n+1}) = bn$ and 
\begin{equation}\label{eq:EvenSnoopy}
\nu_p(U_{2n}) = bn + (a - b) + \nu_p(n) + \lambda_n,
\end{equation}
in which
\begin{equation*}
\lambda_n = 
\begin{cases} 
\nu_p(P'^2 - Q') & \text{if $2 \leq p \leq 3$,\, $2a = b + 1$, and $p \mid n$},\\[3pt]
0 & \text{otherwise}.
\end{cases}
\end{equation*}

\begin{theorem}\label{Theorem:LocalSpecialLucasSGCase3Classifica}
If $p$ is special, $b<2a$, and $r \geq 1$, then
\begin{equation}\label{eq:SplitEvenOdd}
\LL(P,Q,p^{-r})=2T \cup S_{2\ceiling{\frac{r}{b}}}(1;2),
\end{equation}
in which $T$ is the numerical semigroup
\begin{enumerate}[leftmargin=*]
\item $S_{ \ceiling{\frac{r-a+b}{b}}_+}\cup \semigroup{p}$ if $2\leq p\leq 3$,\,  $2a=b+1$, and $P'^2-Q' = 0$,
\item $S_{\lceil \frac{r-a+b}{b} \rceil_+}\cup\{n\in \semigroup{p}:\nu_p(n)\geq r-a+b-\nu_p(P'^2-Q')-bn\} $
if $2\leq p\leq 3$,\, $2a=b+1$, and $P'^2-Q' \neq 0$,
\item $S_{\lceil \frac{r-a+b}{b} \rceil_+}\cup\{n\in \semigroup{p}:\nu_p(n)\geq r-a+b-bn\}$ otherwise.
\end{enumerate}
The Frobenius number of $S$ is $2\big\lceil\frac{r}{b}\big\rceil$ or $2\big\lceil\frac{r}{b}\big\rceil-1$.
Moreover, $\s(S) = \varnothing$ or $\gcd \s(S) \geq 2$.
\end{theorem}

\begin{proof}
Let $S=\LL(P,Q,p^{-r})$.
Since
    \begin{align*}
        \nu_p(U_{2n+1})=bn\geq r &\iff n\geq \frac{r}{b}
        \iff n\geq \Big\lceil \frac{r}{b} \Big\rceil,
    \end{align*}
we get $ S_{2\ceiling{\frac{r}{b}}}(1;2)\subseteq \LL(P,Q,p^{-r})$
and $2\lceil \frac{r}{b}\rceil-1 \notin \LL(P,Q,p^{-r})$.
Next observe that 
\begin{equation*}
\nu_p(U_{2n}) = b(n-1)+a+\nu_p(n)+\lambda_n
\end{equation*}
by \eqref{eq:EvenSnoopy}.  In particular, the odd integers in $S$ are precisely
\begin{equation}\label{eq:OddIntsThere}
\{2\lceil \tfrac{r}{b}\rceil +1, \,\, 2\lceil \tfrac{r}{b}\rceil +3, \,\, 2\lceil \tfrac{r}{b}\rceil +5, \ldots \}
=S_{2\ceiling{\frac{r}{b}}}(1;2). 
\end{equation}

We now consider the even integers in $S$.
If $n \geq \lceil \frac{r}{b}\rceil+1$, then 
\begin{equation*}
\nu_p(U_{2n}) \geq b \Big\lceil \frac{r}{b} \Big\rceil+a+\nu_p(n)+\lambda_n\geq r,
\end{equation*}
so $S_{2\ceiling{ \frac{r}{b} }+1}\subseteq S$.
Since $2\lceil \frac{r}{b}\rceil-1 \notin S$, we have
$g(S) = 2\lceil \frac{r}{b}\rceil$ or $2\lceil \frac{r}{b}\rceil-1$. 
Define 
\begin{equation*}
T =\{n\in \N:2n\in S\}=\{n\in \N:\nu_p(U_{2n})\geq r\},
\end{equation*}
so \eqref{eq:SplitEvenOdd} holds. We claim that $T $
is a numerical semigroup.  Since it contains $0$ and is closed under addition, it suffices to show it
has finite complement in $\N$.

By \eqref{eq:EvenSnoopy} we have 
$\nu_p(U_{2n}) =bn+a-b+\nu_p(n)+\lambda_n\geq bn+a-b$. Therefore,
\begin{align*}
    bn+a-b\geq r
    &\iff bn\geq r-a+b
    \iff n\geq \frac{r-a+b}{b}
    \iff n\geq  \Big\lceil \frac{r-a+b}{b} \Big\rceil_+,
\end{align*}
so $S_{(\lceil \frac{r-a}{b} \rceil +1)_+}\subseteq T $ and hence $T $ has finite complement in $\N$.
If $p \nmid n$, then $\nu_p(n)+\lambda_n=0$ and the above completely characterizes
of the even elements in $S$.

Suppose henceforth that $p\mid n$. Then
\begin{equation}\label{eq:necCondPMidNSpPCase3}
    bn + (a - b) + \nu_p(n) + \lambda_n \geq r 
   \iff  bn+\nu_p(n) \geq r-a+b-\lambda_n.
\end{equation}

\smallskip\noindent\textbf{Case 1.} If $2\leq p\leq 3$,  $ 2a=b+1$, and $ P'^2-Q'=0$, then $\lambda_n=\nu_p(P'^2-Q')=\nu_p(0)=\infty$, so \ref{eq:necCondPMidNSpPCase3} is vacuously true and $\semigroup{p}\subseteq T$. Therefore,
$T=S_{ \ceiling{\frac{r-a+b}{b}}_+}\cup \semigroup{p}$.

\smallskip\noindent\textbf{Case 2.} 
If $2\leq p\leq 3$,  $ 2a=b+1$, and $ P'^2-Q' \neq 0$, then
\begin{align*}
    T \cap \semigroup{p}&=\{n\in \N: bn + (a - b) + \nu_p(n) + \lambda_n \geq r\} \cap \semigroup{p}\\
    &=\{n\in \semigroup{p}:\nu_p(n)\geq r-a+b-\nu_p(P'^2-Q')-bn\},
\end{align*}
so $T$ is the union of $S_{\lceil \frac{r-a+b}{b} \rceil}$ with the set above.

\smallskip\noindent\textbf{Case 3.} Suppose that Cases 1 and 2 do not apply. 
Then $\lambda_n=0$ for all $n \in \N$ and, in a similar manner to the case above, 
\begin{equation*}
    T 
    =S_{\lceil \frac{r-a+b}{b} \rceil}\cup\{n\in \semigroup{p}:\nu_p(n)\geq r-a+b-bn\}.  
\end{equation*}

To finish, observe that $\s(S)$ contains no odd integers. This is because the first odd integer in $S$ is $2\ceiling{\frac{r}{b}}+1$ and $S_{2\ceiling{\frac{r}{b}}+1}\subseteq S$. Therefore, $\s(S)$ contains only even integers. Hence, $2 \mid \gcd \s(S)$.
\end{proof}

\begin{remark}
One can show that the numerical semigroup $T$ in Theorem \ref{Theorem:LocalSpecialLucasSGCase3Classifica} is a 
local Lucas semigroup in the second case of special primes (Theorem \ref{Theorem:LocalSpecialLucasSGCase2Classifica}).
\end{remark}

\section{Classification of dimension-$2$ realizable semigroups}\label{Section:Realizable}

Our aim is to describe semigroups of matricial dimension at most $2$; that is, the semigroups
that are dimension-$2$ realizable.  Recall that $S \subseteq \N$ is such a semigroup if 
$S = \S(A)$ for some $A\in \M_d(\Q)$ with $d \leq 2$.  Since the only semigroups that occur
for $d=1$ are the trivial semigroups $\{0\}$ and $\N$, we focus on $d=2$.  This usage of $d$
should not be confused with the $(2,2)$ entry of the matrix $A$ below.  Recall that $\S(A) \neq \{0\}$
implies that the characteristic polynomial of $A$ has integer coefficients \cite[Thm.~3.1]{NSRM1}.

\begin{proposition}\label{Proposition:General2DimReduction}
Let $A = \tinymatrix{a}{b}{c}{d} \in \M_2(\Q)$ with characteristic polynomial $p_A(x) \in \Z[x]$ and $b \neq 0$.
Let $U_n$ denote the Lucas sequence
$U_n = P U_{n-1} - Q U_{n-2}$,
$U_1 = 1$, and $U_0 = 0$,
in which $P = \tr A$ and $Q = \det A \neq 0$.  Then
\begin{equation*}
\S(A) 
=\LL(P,Q,a)\cap\LL(P,Q,b)\cap\LL\big(P,Q,\tfrac{p_A(a)}{b}\big).
\end{equation*}
\end{proposition}

\begin{proof}
First observe that $U_n \in \Z$ for all $n \in \N$ since $U_0,U_1,P,Q \in \Z$ and 
note that $p_A(x) = x^2 - Px + Q$ \cite[Thm.~10.1.3]{SCLA}, in which $P,Q\in \Z$ by hypothesis.
We use induction to prove that
\begin{equation}\label{eq:MatrixRecurrence}
A^n= U_n A - Q U_{n-1} I
\quad \text{for $n \in \N$}.
\end{equation}
Note that $A^1 = 1 A - 0 I = U_1 A - Q U_0 I$ and, by the Cayley--Hamilton theorem \cite[Thm.~11.2.1]{SCLA},
$A^2 = P A - Q I = U_2 A - Q U_1 I$
since $U_2 = P$.  For our inductive hypothesis, suppose \eqref{eq:MatrixRecurrence} holds for some $n$.
Then it holds for $n+1$ since
\begin{align*}
A^{n+1} 
&= AA^n = A(U_n A - Q U_{n-1} I) \\
&= U_n A^2 - Q U_{n-1} A 
= U_n (P A - Q  I) - Q U_{n-1} A\\
&= (P U_n - Q U_{n-1}) A - U_n Q I 
= U_{n+1} A - U_n Q I.
\end{align*}
Thus,
\begin{equation*}
A^n = U_n \tinymatrix{a}{b}{c}{d} - Q U_{n-1} \tinymatrix{1}{0}{0}{1} 
= \tinymatrix{ U_n a - Q U_{n-1} }{ U_n b }{ U_n c }{ U_n d - Q U_{n-1} }.
\end{equation*}
Since $P = \tr A = a+d \in \Z$ and $Q= \det A = ad-bc \in \Z$, we have
$d = P-a$ and
\begin{equation*}
c = -\frac{Q - ad}{b} = -\frac{Q - a(P-a)}{b} = -\frac{a^2 - P a + Q}{b} = -\frac{p_A(a)}{b}.
\end{equation*}
Since $Q, U_{n-1} \in \Z$, we conclude
\begin{align*}
n \in \S(A) 
&\iff A^n \in \M_2(\Z)
\iff U_n A - Q U_{n-1} I \in \M_2(\Z) \\
&\iff U_n A  \in \M_2(\Z) 
\iff U_n \tinymatrix{a}{b}{c}{d} \in \M_2(\Z)\\
&\iff U_n \Big[ \begin{smallmatrix} a & b \\ -\frac{p_A(a)}{b} & P-a \end{smallmatrix}\Big] \in \M_2(\Z) \\
&\iff U_n a , \, U_n b , \, U_n \tfrac{p_A(a)}{b} \in \Z \\
&\iff n \in \LL(P,Q,a)\cap\LL(P,Q,b)\cap\LL\big(P,Q,\tfrac{p_A(a)}{b}\big).\qedhere
\end{align*}
\end{proof}

\begin{theorem}\label{Theorem:GlobalIff}
An additive semigroup in $\N$ is dimension-$2$ realizable if and only if it is a Lucas semigroup.
\end{theorem}

\begin{proof}
$(\Rightarrow)$
We claim that every dimension-$2$ realizable semigroup in $\N$ is a Lucas semigroup.
Since $\N =  \LL(0,0,1)$ and $\{0\} = \LL(2,0,\frac{1}{3})$ are Lucas semigroups, we can ignore them when they arise.
If $A\in \M_2(\Q)$ is diagonal, then $\S(A) = \{0\}$ or $\N$, so suppose that $A = \tinymatrix{a}{b}{c}{d} \in \M_2(\Q)$ is nondiagonal.  Since $\S(A) = \S(A^{\T})$, we may assume $b \neq 0$.
We can also assume $p_A(x) \in \Z[x]$, since otherwise $\S(A)= \{0\}$ \cite[Thm.~3.1]{NSRM1}.  Let $P = \tr A$ and $Q = \det A$, so that $p_A(x) = x^2 - Px + Q$.
\begin{enumerate}[leftmargin=*]
\item If $P\neq 0$ and $Q = 0$, then $A^n = P^{n-1} A$ for $n\geq 1$.
Thus, $\S(A) = \{0\}$, $\N$, or $S_n = \LL(2,0,\frac{1}{2^{n-1}})$ (Theorem~\ref{Theorem:PQZero}) for some $n \geq 2$. 

\item If $P = Q = 0$, then Cayley--Hamilton and induction ensure that $A^n = 0$ for $n\geq 2$.  Thus,
$\S(A) = \N$ or $\S(A) = \semigroup{2,3} = \LL(0,0,\frac{1}{2})$.

\item If $Q \neq 0$, then Proposition~\ref{Proposition:General2DimReduction} and Lemma~\ref{Lemma:FixedPQClosure} 
provide an $R \in \Q$ such that
\begin{equation*}
        \S(A)=\LL(P,Q,a)\cap \LL(P,Q,b)\cap \LL\big(P,Q,\tfrac{p_A(a)}{b}\big) =\LL(P,Q,R). 
\end{equation*}
\end{enumerate}    
Thus, every dimension-$2$ realizable semigroup is a Lucas semigroup.

\smallskip\noindent$(\Leftarrow)$
Let $\LL(P,Q,R)$ be a Lucas semigroup.
If $R \in \Z$, then $\LL(P,Q,R)=\N$, which has matricial dimension $1$. 
Thus, we assume $R \notin \Z$.
If $PQ = 0$, then Corollary~\ref{Corollary:PQZero} ensures $\LL(P,Q,R)$ is dimension-$2$ realizable. 
Hence, we assume $PQ\neq 0$.   Write $R = N/D$, in which $N,D \in \Z$ with $D \neq 0$ and $\gcd(N,D)=1$,
and define
$A = \tinymatrix{0}{1/D}{ - QD}{P }$.
In the notation of Proposition~\ref{Proposition:General2DimReduction},
    \begin{align*}
    \S(A)&=\LL(P,Q,a)\cap \LL(P,Q,b)\cap \LL\big(P,Q,\tfrac{p_A(a)}{b}\big)\\
    &=\LL(P,Q,0)\cap \LL\big(P,Q,\tfrac{1}{D}\big) \cap \LL(P,Q,QD)\\
    &=\N \cap \LL(P,Q,R) \cap \N=\LL(P,Q,R).
    \end{align*}
Thus, every Lucas semigroup is dimension-$2$ realizable.    
\end{proof}

\section{Applications}\label{Section:Applications}

A natural question is whether every Lucas semigroup a local Lucas semigroup.
The wide variety of local Lucas semigroups described in the previous sections 
suggest that the problem is nontrivial.  However, our classification of local Lucas semigroups
permits us to answer this in the negative.

\begin{theorem}\label{Theorem:NotGlobal}
    There exists a Lucas semigroup that is not a local Lucas semigroup.
\end{theorem}

\begin{proof}
    Let $P=18$ and $Q=8$. Then
    $\LL(P,Q, 3^{-1})=\semigroup{2}$ 
    (Theorem \ref{Theorem:RegPrimeCyclicLocal}) and
    $\LL(P,Q, 2^{-5})= S_6$ 
     (Theorem \ref{Theorem:LocalSpecialLucasSGCase1Classifica})
     are local Lucas semigroups.  Their intersection, the Lucas semigroup
     $\LL(P,Q, 3^{-1}2^{-5})=\semigroup{6,8,10}$,
     is not local since it is not cyclic (Theorem \ref{Theorem:RegPrimeCyclicLocal}), 
     numerical (Theorem \ref{Theorem:SpecialNumerical}), 
     of a form that occurs in Theorem \ref{Theorem:PQZero}. 
\end{proof}

If $S \subseteq \N$ is a semigroup with $\mdim S = 1$, then $S = \{0\}$ or $\N$ \cite[Prop.~2.1]{NSRM1}.
Consequently, the theorems above provide the following hierarchy:
\begin{align*}
&\{ \text{local Lucas semigroups} \}
\subsetneq \{ \text{Lucas semigroups} \}\\
&\qquad= \{ \text{Dim.-$2$ realizable semigroups}\} \\
&\qquad= \{ \text{Mat.~dim.~$2$ semigroups}\} \cup \{ \text{Mat.~dim.~$1$ semigroups}\} \\
&\qquad= \{ \text{Mat.~dim.~$2$ semigroups}\} \cup \{ \{0\}, \N\}.
\end{align*}

Recall from the introduction the following conjecture, already implicit in \cite{NSRM1}
and circulated informally among researchers in the field.

\begin{conjecture}[Lonely element conjecture]\label{Problem:Plus}
If every nonzero small element of a numerical semigroup is lonely, must the semigroup be
dimension-$2$ realizable?
\end{conjecture}

A numerical semigroup $S$ is \emph{$\PPM$ avoiding} if $n,n+1 \in S$ implies $n+2 \in S$; in this case, induction yields $S_n \subseteq S$.  A $\PPM$ avoiding numerical semigroup has $\mdim S \geq 2$ \cite[Thm.~5.1]{NSRM1}.
For example, $\semigroup{3,5,7} = \{0,3,5,6,7,8,\ldots\}$ is $\PPM$ avoiding and has matricial dimension $2$, while
$\semigroup{3,4} = \{0,3,4,6,7,8,\ldots\}$ is not $\PPM$ avoiding and $\mdim \semigroup{3,4} = 3$ \cite[Table 1]{NSRM1}.
The lonely element conjecture is true if and only if every $\PPM$ avoiding semigroup is dimension-$2$ realizable.

Prior to this paper, the semigroups of matricial dimension $2$ were not classified.  When presented with a $\PPM$ avoiding semigroup, we had no general method to show its matricial dimension was greater than $2$.  The next result remedies this.

\begin{theorem}\label{Theorem:SmallSet}
Let $S$ be a numerical Lucas semigroup.  Then $\s(S) = \varnothing$ or $\gcd \s(S) \geq 2$.
If neither condition holds, then $S$ is not dimension-$2$ realizable.
\end{theorem}

\begin{proof}
Since a numerical Lucas semigroup is the intersection of local numerical Lucas semigroups, 
it suffices to prove the result for a local numerical Lucas semigroup.  
In light of Theorems \ref{Theorem:PQZero}, \ref{Theorem:RegPrimeCyclicLocal}, and
\ref{Theorem:SpecialNumerical}, a local Lucas semigroup is numerical if and only if $PQ \neq 0$ and $p$ is special.
This breaks down into the three cases from Subsections \ref{Subsection:Special1}, \ref{Subsection:Special2}, and \ref{Subsection:Special3}.
In Case 1, Theorem \ref{Theorem:LocalSpecialLucasSGCase1Classifica} ensures that $s(\S) = \varnothing$.
Corollary \ref{Corollary:Special2H} and Theorem \ref{Theorem:LocalSpecialLucasSGCase3Classifica}
handle Cases 2 and 3, respectively.
\end{proof}

For example, 
\begin{equation}\label{eq:5716}
S= \semigroup{5, 7, 16, 18} = \{0, 5, 7, 10, 12, 14, 15, 16, 17, 18,\ldots\}
\end{equation}
has $\s(S) = \{5,7,10,12\} \neq \varnothing$ and $\gcd \s(S) = 1 < 2$, so
Theorem \ref{Theorem:SmallSet} ensures it is not dimension-$2$ realizable.  This proves the next result:

\begin{theorem}\label{Theorem:Plus}
The lonely element conjecture is false:
there exists a $\PPM$ avoiding semigroup with matricial dimension greater than $2$.  
\end{theorem}

\begin{proof}
The semigroup $S$ in  \eqref{eq:5716} is $\PPM$ avoiding and $\mdim S >2$.
\end{proof}

\section{Further directions}\label{Section:Open}
We conclude with several questions spurred by the work above.   The first question stems from \cite{NSRM1,NSRM2},
which introduced the concept of matricial dimension.

\begin{problem}\label{Problem:A}
Find a method to compute the matricial dimension of a semigroup.
\end{problem}

In particular, what is $\mdim \semigroup{5, 7, 16, 18}$?  The proof of Theorem \ref{Theorem:Plus}
ensures that it is greater than $2$.  Since the matricial dimension of a numerical semigroup cannot exceed its multiplicity
\cite[Cor.~3.6]{NSRM2}, we are left with
\begin{equation*}
3 \leq \mdim \semigroup{5, 7, 16, 18} \leq 5.
\end{equation*}
A more modest goal than Problem \ref{Problem:A} is:

\begin{problem}
Find $\mdim \semigroup{5, 7, 16, 18}$.
\end{problem}

The answer is at least $3$, but it could be as high as $5$.  

Two questions directly inspired by the work above are the following:

\begin{problem}
Characterize the $\PPM$ avoiding semigroups.
\end{problem}

\begin{problem}
Characterize the dimension-$3$ realizable semigroups.
\end{problem}

The methods used in this paper would require significant effort to generalize.
First, the Cayley--Hamilton approach employed in Proposition \ref{Proposition:General2DimReduction} and 
Theorem \ref{Theorem:GlobalIff} may prove more cumbersome, although it might still be possible to characterize
dimension-$3$ realizable semigroups via the divisibility properties of three-term recurrences.  However, this would
require knowledge about the $p$-adic valuations of more complicated relatives of the Lucas numbers.
To our knowledge, a complete exploration of that topic, analogous to \cite{BallotPaper,BallotBook}, has not been undertaken. 

One can consider membership patterns beyond $\PPM$, such as $\scriptstyle +++-$.  However, their consideration seems premature.  In another direction, a question about Lucas sequences that arose in our investigations is:

\begin{problem}
For each prime $p$ and $n,k \geq 1$, is there a Lucas sequence such that $\rho(p)=n$ and $\nu_p(U_{\rho(p)})=k$?
\end{problem}

A positive answer would permit the construction of all dimension-$2$ realizable semigroups.
At present, the bottleneck appears to be number-theoretic.


\bibliography{NSRM3}
\bibliographystyle{amsplain}

\end{document}